\documentclass[10pt,reqno]{amsart}

\usepackage{amsmath,amssymb,amsthm,graphicx,xcolor}

\theoremstyle{plain}
\newtheorem{thm}{Theorem}
\newtheorem{prop}[thm]{Proposition}

\theoremstyle{definition}
\newtheorem{rem}[thm]{Remark}
\newtheorem*{rem*}{Remark}

\begin{document}

\title{Applications of Gr\"unbaum-type inequalities}

\author{Matthew Stephen}
\address{Matthew Stephen, Department of Mathematical \& Statistical Sciences, University of Alberta, Edmonton, Alberta, T6G 2G1, Canada}
\email{mastephe@ualberta.ca}

\author{Vlad Yaskin}
\address{Vlad Yaskin, Department of Mathematical \& Statistical Sciences, University of Alberta, Edmonton, Alberta, T6G 2G1, Canada}
\email{yaskin@ualberta.ca}

\thanks{Both authors were supported in part by NSERC}



\begin{abstract}
Let $1\leq i \leq k < n$ be integers. We prove the following exact inequalities for any convex body $K\subset\mathbb{R}^n$ with centroid at the origin, and any $k$-dimensional subspace $E\subset \mathbb{R}^n$:
\begin{align*}
&V_i \big( K\cap E \big) \geq \left( \frac{i+1}{n+1} \right)^i
	\max_{x\in K} V_i \big( ( K-x) \cap E \big) , \\
&\widetilde{V}_i \big( K\cap E \big) \geq \left( \frac{i+1}{n+1} \right)^i
	\max_{x\in K} \widetilde{V}_i \big( ( K-x) \cap E \big) ;
\end{align*}
$V_i$ is the $i$th intrinsic volume, and $\widetilde{V}_i$ is the $i$th dual volume taken within $E$. Our results are an extension of an inequality of M. Fradelizi, which corresponds to the case $i=k$. Using the same techniques, we also establish extensions of ``Gr\"unbaum's inequality for sections" and ``Gr\"unbaum's inequality for projections" to dual volumes.
\end{abstract}

\maketitle

\section{Introduction}

A convex body $K\subset\mathbb{R}^n$ is a convex and compact subset of $\mathbb{R}^n$ with non-empty interior. The centroid of $K$ is the affine covariant point
\begin{align*}
g(K) := \frac{ 1 }{ \mbox{vol}_n(K) } \int_K x\, dx \in\mbox{int}(K) .
\end{align*}
Makai and Martini conjectured the following (Conjecture 3.3 in \cite{MM}): for integers $1\leq k < n$, any convex body $K\subset\mathbb{R}^n$ with centroid at the origin, and any $k$-dimensional subspace $E\in G(n,k)$,
\begin{align}\label{MM Conjecture}
\mbox{vol}_k (K\cap E ) \geq \left( \frac{ k+1 }{ n+1 } \right)^k \max_{x\in\mathbb{R}^n}
	\mbox{vol}_k \big( (K-x)\cap E \big) .
\end{align}
Here, $\mbox{vol}_k$ denotes $k$-dimensional Lebesgue volume. They were able to prove (\ref{MM Conjecture}) for $k=1,\, n-1$; see Theorem 3.1 and Proposition 3.4 in \cite{MM}. Shortly thereafter, Fradelizi \cite{F} proved the conjecture for all $k$, including sharpness and a complete characterization of the equality conditions.

In this paper, we generalize (\ref{MM Conjecture}) to intrinsic and dual volumes. We refer the reader to \cite{Ga} for a nice summary of these concepts, but let us recall the basic definitions. For a convex and compact set $L\subset\mathbb{R}^n$ and the $n$-dimensional Euclidean ball $B_2^n$ with unit radius, Steiner's formula expands the volume of the Minkowski sum $L + tB_2^n$ into a polynomial of $t$:
\begin{align*}
\mbox{vol}_n \big( L + t B_2^n \big)
	= \sum_{i=0}^n \kappa_{n-i} V_i(L) t^{n-i} \qquad \forall \ t\geq 0.
\end{align*}
The coefficient $V_i(L)$ is the $i$th intrinsic volume of $L$, and $\kappa_{n-i}$ denotes the $(n-i)$-dimensional volume of $B_2^{n-i}$. We prove the following:

\begin{thm}\label{MainThm1}
Consider integers $1\leq i \leq k < n$. Let $K\subset\mathbb{R}^n$ be a convex body with centroid at the origin, and let $E\in G(n,k)$. Then
\begin{align}\label{main ineq}
V_i(K\cap E ) \geq \left( \frac{ i+1 }{ n+1 } \right)^i
	\max_{x\in K} V_i\big( (K-x) \cap E \big) .
\end{align}
The constant in this inequality is the best possible.
\end{thm}

\begin{rem*}
When $i=k$, inequality (\ref{main ineq}) yields (\ref{MM Conjecture}). When $i = k-1$, our inequality gives a lower bound for the $k-1$ dimensional surface area of a $k$-dimensional section of $K$ through its centroid.
\end{rem*}

A star body $L\subset\mathbb{R}^n$ is a non-empty compact subset of $\mathbb{R}^n$ which is star-shaped with respect to the origin, and whose radial function
\begin{align*}
\rho_L(\xi) := \max \lbrace a \geq 0 \, | \, a\xi\in L \rbrace , \qquad \xi\in S^{n-1},
\end{align*}
is positive and continuous. The radial sum of the star body $L\subset\mathbb{R}^n$ with the ball $tB_2^n$ of radius $t>0$ is the star body $L \widetilde{+} tB_2^n$ whose radial function is equal to $\rho_L(\xi) + t$ for all $\xi\in S^{n-1}$. The dual Steiner's formula expands the volume of $L \widetilde{+} tB_2^n$ into a polynomial of $t$:
\begin{align*}
\mbox{vol}_n \big( L + t B_2^n \big)
	= \sum_{i=0}^n \genfrac(){0pt}{0}{n}{i} 	
	\widetilde{V}_i(L) t^{n-i} \qquad \forall \ t\geq 0.
\end{align*}
The coefficient $\widetilde{V}_i(L)$ is the $i$th dual volume of $L$. We prove the following:

\begin{thm}\label{MainThm2}
Consider integers $1\leq i \leq k < n$. Let $K\subset\mathbb{R}^n$ be a convex body with centroid at the origin, and let $E\in G(n,k)$. Then
\begin{align}\label{main ineq 2}
\widetilde{V}_i \big( K\cap E \big) \geq \left( \frac{i+1}{n+1} \right)^i
	\max_{x\in K} \widetilde{V}_i \big( ( K-x ) \cap E \big) ,
\end{align}
where the dual volumes are taken within the $k$-dimensional subspace $E$. The constant in this inequality is the best possible.
\end{thm}

\begin{rem*}
When $i=k$, inequality (\ref{main ineq 2}) yields (\ref{MM Conjecture}).
\end{rem*}

Essentially, we prove Theorem \ref{MainThm1} and Theorem \ref{MainThm2} as consequences of ``Gr\"unbaum's inequality for sections": for integers $1\leq k\leq n$, a convex body $K\subset\mathbb{R}^n$ with centroid at the origin, and $E\in G(n,k)$,
\begin{align}\label{Gruenbaum_sec}
\mbox{vol}_k (K\cap E \cap \xi^+ )
	\geq \left( \frac{k}{n+1}\right)^k \mbox{vol}_k (K\cap E)
	\quad \mbox{for all} \quad \xi\in S^{n-1}\cap E.
\end{align}
Here, $\xi^+ := \lbrace x\in\mathbb{R}^n\, | \, \langle x,\xi\rangle\geq 0\rbrace$. Inequality (\ref{Gruenbaum_sec}) was proved in  \cite{MSZ}. The reader is also referred to the papers  \cite{FMY} and \cite{MNRY} for previous results on this topic. Inequality (\ref{Gruenbaum_sec}) implies ``Gr\"unbaum's inequality for projections",
\begin{align}\label{Gruenbaum_pro}
\mbox{vol}_k \big( (K|E) \cap \xi^+ \big)
	\geq \left( \frac{k}{n+1}\right)^k \mbox{vol}_k (K|E)
	\quad \mbox{for all} \quad \xi\in S^{n-1}\cap E,
\end{align}
which was proved earlier in \cite{SZ} using a different method. The case $k=n$ in both (\ref{Gruenbaum_sec}) and (\ref{Gruenbaum_pro}) is Gr\"unbaum's classic inequality \cite{Gr}, which states
\begin{align}\label{Gruenbaum_classic}
\mbox{vol}_n (K \cap \xi^+ ) \geq \left( \frac{n}{n+1}\right)^n \mbox{vol}_n (K) \quad \mbox{for all} \quad \xi\in S^{n-1}
\end{align}
for every convex body with centroid at the origin.

In this paper, we also prove an analogue of (\ref{Gruenbaum_sec}) and (\ref{Gruenbaum_pro}) for dual volumes.

\begin{thm}\label{MainThm3}
Consider integers $1\leq i \leq k\leq n$. Let $K\subset\mathbb{R}^n$ be a convex body with centroid at the origin, and let $E\in G(n,k)$. Then
\begin{align}
\widetilde{V}_i (K \cap E \cap \xi^+ )
	&\geq \left( \frac{i}{n+1}\right)^i \widetilde{V}_i (K \cap E)
	\label{main ineq 3.1} \\
\mbox{and} \qquad \widetilde{V}_i \big( (K|E) \cap \xi^+ \big)
	&\geq \left( \frac{i}{n+1}\right)^i \widetilde{V}_i (K|E) \label{main ineq 3.2}
\end{align}
for all $\xi\in S^{n-1}\cap E$, where the dual volumes are taken within the $k$-dimensional subspace $E$. The constant in each inequality is the best possible.
\end{thm}

\begin{rem*}
When $i=k$, inequality (\ref{main ineq 3.1}) yields (\ref{Gruenbaum_sec}) and inequality (\ref{main ineq 3.2}) yields (\ref{Gruenbaum_pro}). Inequality (\ref{main ineq 3.1}) can also be written in the form
\begin{align*}
\int_{S^{n-1} \cap E \cap \xi^+} \rho_K(u)^i \, du
	\geq \left( \frac{i}{n+1}\right)^i \int_{S^{n-1} \cap E } \rho_K(u)^i \, du ;
\end{align*}
see (\ref{dual_vol}) in Section 2.
\end{rem*}

The paper is organized as follows. We present some preliminaries in Section 2, prove Theorem \ref{MainThm1} in Section 3, prove Theorem \ref{MainThm2} in Section 4, and prove Theorem \ref{MainThm3} in Section 5.

\section{Preliminaries}

According to Kubota's integral formula (e.g. equation A.47 in \cite{Ga}), the $i$th intrinsic volume $V_i(L)$ of a convex and compact $L\subset\mathbb{R}^n$ is essentially the average $i$-dimensional volume of the orthogonal projection $L|F$ taken over all $i$-dimensional subspaces $F\in G(n,i)$:
\begin{align*}
V_i(L) = \frac{ \kappa_n }{ \kappa_i \kappa_{n-i} } \genfrac(){0pt}{0}{n}{i} 		
	\int_{G(n,i)} \mbox{vol}_i \big( L|F \big) \, dF .
\end{align*}
Here, we are integrating with respect to the unique Haar probability measure on the Grassmannian $G(n,i)$ of $i$-dimensional subspaces of $\mathbb{R}^n$.

Similarly, the dual Kubota integral formula (e.g. Theorem A.7.2 in \cite{Ga}) asserts that the $i$th dual volume $\widetilde{V}_i(L)$ of a star body $L\subset\mathbb{R}^n$ is the average $i$-dimensional volume of the section $L\cap F$ taken over all $F\in G(n,i)$:
\begin{align*}
\widetilde{V}_i(L) = \frac{ \kappa_n }{ \kappa_i }
	\int_{G(n,i)} \mbox{vol}_i \big( L\cap F \big) \, dF .
\end{align*}
The dual volume $\widetilde{V}_i(L)$ can also be expressed as follows: 
\begin{align}\label{dual_vol}
\widetilde{V}_i(L) = \frac{1}{n} \int_{S^{n-1}} \rho_L(u)^i \, du
	= \frac{i}{n} \int_L |x|^{-n+i} \, dx .
\end{align}

The following proposition is a simple consequence of Gr\"unbaum's inequality for sections.

\begin{prop}\label{Gruenbaum}
Let $K\subset\mathbb{R}^n$ be a convex body with centroid at the origin. Let $1\leq i \leq k\leq n$ be integers. For every $E\in G(n,k)$, $F\in G(n,i)$ with $F\subset E$, and $\xi\in S^{n-1}\cap F$, we have
\begin{align}
&\mbox{vol}_i \Big( \big( (K \cap E ) | F \big) \cap \xi^+ \Big)
	\geq \left( \frac{i}{n+1} \right)^i \mbox{vol}_i \big( (K \cap E ) | F \big) \label{inequality1} \\
\mbox{and} \quad &\mbox{vol}_i \big( (K | E ) \cap F \cap \xi^+ \big)
	\geq \left( \frac{i}{n+1} \right)^i \mbox{vol}_i \big( (K | E ) \cap F \big) . \label{inequality2}
\end{align}
These inequalities are exact. For example, there is equality in both inequalities when
\begin{align}\label{equality}
K = \mbox{conv} \left( r_0 B_2^{i-1} - a \left( \frac{n-i+1}{n+1}\right) \xi, r_1 B_2^{n-i} + a \left( \frac{i}{n+1} \right) \xi \right) ;
\end{align}
$B_2^{i-1}$ is the unit Euclidean ball in $F\cap\xi^\perp$ centred at the origin, $B_2^{n-i}$ is the unit Euclidean ball in $F^\perp$ centred at the origin, and $a, r_0, r_1>0$ are any constants.
\end{prop}

\begin{proof}
When $i=k$ or $k=n$, the inequalities (\ref{inequality1}) and (\ref{inequality2}) are exactly Gr\"unbaum's inequality for sections (\ref{Gruenbaum_sec}) and Gr\"unbaum's inequality for projections (\ref{Gruenbaum_pro}).

Assume $i<k<n$. Observe that
\begin{align*}
( K \cap E ) | F = \Big( K | \big( F \oplus E^\perp \big) \Big) \cap F .
\end{align*}
Let $K_1$ be the $(k-i)$-symmetral of $K$ parallel to $\big( F \oplus E^\perp \big)^\perp$, and let $K_2$ be the $(n-k)$-symmetral of $K$ parallel to $E^\perp$; i.e.
\begin{align*}
K_1 &= \bigcup_{x\in K | ( F \oplus E^\perp )} \left(
	\left( \frac{ \mbox{vol}_{k-i} \left( (K-x)\cap \big( F \oplus E^\perp \big)^\perp \right) }{ \kappa_{k-i} } \right)^\frac{1}{k-i}
	B_2^{k-i} + x \right) \\
\mbox{and} \quad K_2 &= \bigcup_{x\in K | E} \left(
	\left( \frac{ \mbox{vol}_{n-k} \left( (K-x)\cap E^\perp \right) }{ \kappa_{n-k} } \right)^\frac{1}{n-k}
	B_2^{n-k} + x \right) ,
\end{align*}
where $B_2^{k-i}$ and $B_2^{n-k}$ are the Euclidean balls in $\big( F\oplus E^\perp \big)^\perp$ and $E^\perp$, respectively, with unit radius and centres at the origin. Now, $K_1$ and $K_2$ are convex bodies in $\mathbb{R}^n$ with
\begin{align*}
K | \big( F \oplus E^\perp \big) = K_1 | \big( F \oplus E^\perp \big)
	= K_1 \cap \big( F \oplus E^\perp \big) ,
	\qquad K | E = K_2 | E = K_2 \cap E ,
\end{align*}
and centroids at the origin. Therefore,
\begin{align*}
&(K\cap E) | F = \Big( K | \big( F \oplus E^\perp \big) \Big) \cap F
	= K_1 \cap \big( F \oplus E^\perp \big) \cap F = K_1 \cap F \\
\mbox{and} \qquad & (K|E)\cap F = K_2 \cap E \cap F = K_2 \cap F .
\end{align*}
We now see that (\ref{inequality1}) and (\ref{inequality2}) follow from an application of Gr\"unbaum's inequality for sections to $K_1$ and $K_2$, respectively, and the subspace $F$.

Finally, to show that inequalities (\ref{inequality1}) and (\ref{inequality2}) are sharp, suppose $K$ has the form (\ref{equality}). We first show that the centroid of $K$ is at the origin. Indeed, by symmetry, $g(K)$ must lie on the line $\mathbb{R}\xi$ passing through the origin and parallel to $\xi$. For $t\in \left[ -a\left(\frac{n-i+1}{n+1}\right), a\left(\frac{i}{n+1}\right)\right]$, the section $K\cap \lbrace t\xi + \xi^\perp \rbrace$ is the product of balls
\begin{align*}
\left( \widetilde{r}_0(t) B_2^{i-1} \right) \times
	\left( \widetilde{r}_1(t) B_2^{n-i} \right)
\end{align*}
where
\begin{align*}
\widetilde{r}_0(t) = r_0 \left( \frac{i}{n+1} - \frac{t}{a} \right) \quad \mbox{and}
	\quad \widetilde{r}_1(t) = r_1 \left( \frac{n-i+1}{n+1} + \frac{t}{a} \right).
\end{align*}
Applying Fubini's Theorem and a change of variables,
\begin{align*}
&\int_K \langle x, \xi\rangle \, dx
	= \int_{-a\left(\frac{n-i+1}{n+1}\right)}^{a\left(\frac{i}{n+1}\right)}
	\int_{ \widetilde{r}_0(x_1) B_2^{i-1}} \int_{ \widetilde{r}_1(x_1) B_2^{n-i} }
	x_1 \, dx_n \cdots dx_2 \, dx_1 \\
&= r_0^{i-1} r_1^{n-i} \kappa_{i-1} \kappa_{n-i}
	\int_{-a\left(\frac{n-i+1}{n+1}\right)}^{a\left(\frac{i}{n+1}\right)}
	\left( \frac{i}{n+1} - \frac{x_1}{a}\right)^{i-1}
	\left( \frac{n-i+1}{n+1} + \frac{x_1}{a}\right)^{n-i} x_1 \, dx_1 \\
&= a^2 r_0^{i-1} r_1^{n-i} \kappa_{i-1} \kappa_{n-i}
	\int_0^1 (1-t)^{i-1} t^{n-i} \left( t - \frac{n-i+1}{n+1} \right) \, dt .
\end{align*}
The last integral is equal to
\begin{align*}
&\int_0^1 t^{n-i+1} (1-t)^{i-1} \, dt - \left( \frac{n-i+1}{n+1}\right)
	\int_0^1 t^{n-i} (1-t)^{i-1} \, dt \\
&= \frac{ \Gamma(n-i+2) \Gamma(i) }{ \Gamma(n+2) }
	- \left( \frac{n-i+1}{n+1}\right) \frac{ \Gamma(n-i+1) \Gamma(i)}{\Gamma(n+1)} \\
&= 0  ,
\end{align*}
using well-known identities for the Gamma function. Therfore, $\langle g(K), \xi\rangle = 0$, implying $g(K)$ is the origin. Now, we have $K\cap E = K|E$, so
\begin{align*}
(K \cap E)|F = K|F \qquad \mbox{and} \qquad (K|E)\cap F = K\cap F .
\end{align*}
Finally, observing that
\begin{align*}
K|F = K\cap F
	= \mbox{conv} \left( r_0 B_2^{i-1} - a \left( \frac{n-i+1}{n+1}\right) \xi,
	a \left( \frac{i}{n+1} \right) \xi \right)
\end{align*}
is an $i$-dimensional cone in $F$ whose base is orthogonal to $\xi$, a simple calculation verifes that $K$ gives equality in (\ref{inequality1}) and (\ref{inequality2}).
\end{proof}

\begin{rem}\label{Grem}
The inequalities in Proposition \ref{Gruenbaum} are equivalent to
\begin{align*}
\mbox{vol}_i \Big( \big( (K \cap E ) | F \big) \cap \xi^+ \Big)
	&\geq \left( \frac{ i^i }{ (n+1)^i - i^i } \right) \mbox{vol}_i \Big( \big( (K \cap E ) | F \big) \cap \xi^- \Big) \\
\mbox{and} \qquad \mbox{vol}_i \big( (K | E ) \cap F \cap \xi^+ \big)
	&\geq \left( \frac{ i^i }{ (n+1)^i - i^i } \right) \mbox{vol}_i \big( (K | E ) \cap F \cap \xi^- \big) ,
\end{align*}
where $\xi^- := \lbrace x\in\mathbb{R}^n\, | \, \langle x,\xi\rangle\leq 0\rbrace$.
\end{rem}

\section{Intrinsic Volumes of Sections}

\begin{proof}[Proof of Theorem \ref{MainThm1}]
By the continuity and translation invariance of intrinsic volumes, there is an $x_0 \in E^\perp$ such that
\begin{align*}
V_i \big( (K-x_0) \cap E \big)
	&= \max_{x\in E^\perp} V_i \big( (K-x) \cap E \big) \\
	&= \max_{x\in\mathbb{R}^n} V_i \big( (K-x) \cap E \big)
	= \max_{x\in K}  V_i \big( (K-x) \cap E \big) .
\end{align*}
Without loss of generality, $x_0$ is not the origin. Let $\xi\in S^{n-1}$ be the unique unit vector parallel to $x_0$ and such that $-t_0 := \langle \xi,x_0 \rangle <0$. For any subspace $F$ of $\mathbb{R}^n$, define
\begin{align*}
F_\xi := \mbox{span}(F,\xi) .
\end{align*}
Let $G(E,i)$ denote the Grassmannian of $i$-dimensional subspaces of $E$. Note that
\begin{align*}
\Big( (K-t\xi)\cap E \Big) | F = \Big( \big( (K \cap E_\xi)| F_\xi \big) \cap \lbrace t\xi + F \rbrace \Big) - t\xi
	\qquad \forall \ F\in G(E,i) ,
\end{align*}
because $\xi\in E^\perp$. So, by Kubota's integral formula, we have
\begin{align}
V_i \big( (K-x_0) \cap E \big) &= c_{k,i} \int_{G(E,i)} \mbox{vol}_i \Big( \big( (K \cap E_\xi)| F_\xi \big)
	\cap \lbrace -t_0\xi + F\rbrace \Big) \, dF \label{average1} \\
\mbox{and} \qquad V_i \big( K\cap E \big) &= c_{k,i} \int_{G(E,i)}
	\mbox{vol}_i \Big( \big( (K \cap E_\xi)| F_\xi \big) \cap F \Big) \, dF \label{average2} ,
\end{align}
where $c_{k,i}>0$ is a constant depending on $k$ and $i$.

Consider any $F\in G(E,i)$ for which
\begin{align*}
\mbox{vol}_i \Big( \big( (K \cap E_\xi)| F_\xi \big) \cap \lbrace -t_0\xi + F \rbrace \Big)
	> \mbox{vol}_i \Big( \big( (K \cap E_\xi)| F_\xi \big) \cap F \Big) .
\end{align*}
It is possible that
\begin{align*}
\mbox{vol}_i \Big( \big( (K \cap E_\xi)| F_\xi \big) \cap \lbrace -t_0\xi + F \rbrace \Big)
	< \max_{t\in\mathbb{R}} \mbox{vol}_i \Big( \big( (K \cap E_\xi)| F_\xi \big) \cap \lbrace t \xi + F \rbrace \Big) ,
\end{align*}
but this does not matter to the following argument. Let $\widetilde{K}$ be the $i$-symmetral of $(K \cap E_\xi)| F_\xi$ in $F_\xi$ parallel to $F$; i.e.
\begin{align*}
\widetilde{K}
	= \bigcup_{ t\xi\in (K \cap E_\xi)| \xi } \left(
	\left( \frac{ \mbox{vol}_i \Big( \big( (K \cap E_\xi)| F_\xi \big) \cap \lbrace t\xi + F \rbrace \Big) }{\kappa_i } \right)
	^\frac{1}{i} B_2^i + t\xi \right) ,
\end{align*}
where $B_2^i$ is the $i$-dimensional Euclidean ball in $F$ with unit radius and centred at the origin. Note that $\widetilde{K}$ is an $(i+1)$-dimensional convex body in $F_\xi$. Let $G$ be the unique $(i+1)$-dimensional cone in $F_\xi$ with
\begin{align}\label{cone}
\begin{array}{ll}
\bullet \ \ i\mbox{-dimensional base } \widetilde{K} \cap \lbrace -t_0\xi + F \rbrace ; \\
\bullet \ \ i\mbox{-dimensional cross-section } G\cap F = \widetilde{K}\cap F ;
\end{array}
\end{align}
see Figure \ref{RatioCone}.
\begin{figure}[ht]
\includegraphics[width=\textwidth]{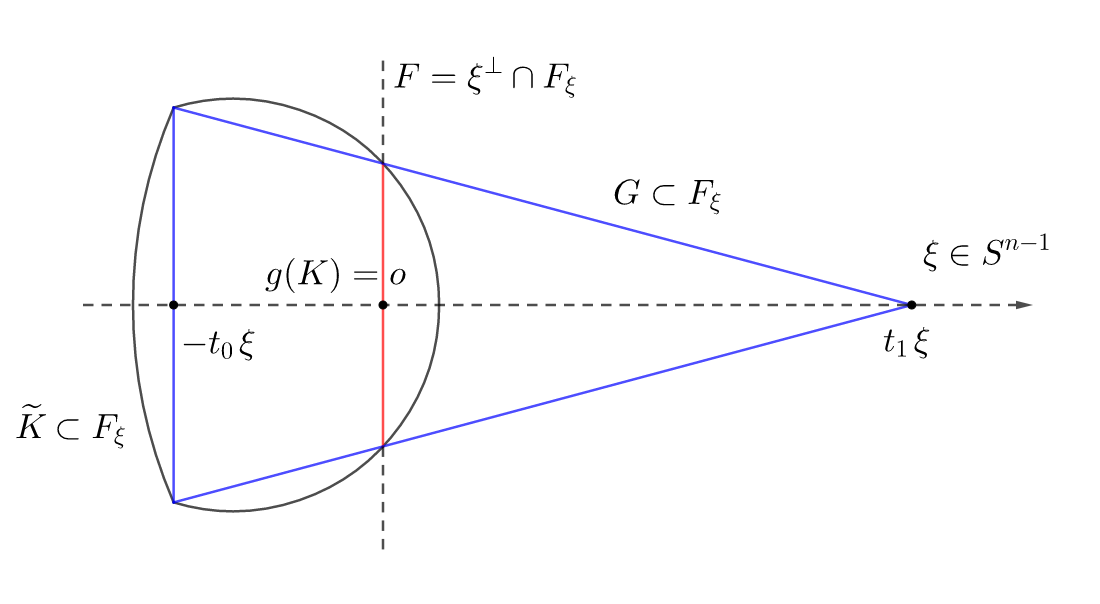}
\caption{The $i$-symmetral $\widetilde{K}$ and the cone $G$, both in $F_\xi$.} \label{RatioCone}
\end{figure}
Necessarily, the apex of the cone $G$ is at $t_1\xi$ for some $t_1>0$. By convexity, we have that $G\cap\xi^- \subset \widetilde{K}\cap\xi^-$ and $\widetilde{K}\cap\xi^+\subset G\cap\xi^+$. Therefore,
\begin{align*}
&\mbox{vol}_{i+1} \big( G\cap\xi^- \big) \leq \mbox{vol}_{i+1} \big( \widetilde{K}\cap\xi^- \big)
	= \mbox{vol}_{i+1} \Big( \big( (K \cap E_\xi)| F_\xi \big) \cap\xi^- \Big) \\
\mbox{and} \quad &\mbox{vol}_{i+1} \Big( \big( (K \cap E_\xi)| F_\xi \big) \cap\xi^+ \Big)
	= \mbox{vol}_{i+1} \big( \widetilde{K}\cap\xi^+ \big) \leq \mbox{vol}_{i+1} \big( G\cap\xi^+ \big) ,
\end{align*}
implying
\begin{align}\label{contra1}
\frac{ \mbox{vol}_{i+1} \Big( \big( (K \cap E_\xi)| F_\xi \big) \cap\xi^+ \Big) }
	{ \mbox{vol}_{i+1} \Big( \big( (K \cap E_\xi)| F_\xi \big) \cap\xi^- \Big) }
	\leq \frac{ \mbox{vol}_{i+1} \big( G\cap\xi^+ \big) }{ \mbox{vol}_{i+1} \big( G\cap\xi^- \big) }  .
\end{align}
Define
\begin{align}\label{L0}
L_0 := \mbox{conv} \left( \left( \frac{ \mbox{vol}_i \big( \widetilde{K} \cap \lbrace -t_0\xi + F \rbrace \big) }{\kappa_i } \right)
	^\frac{1}{i} B_2^i - t_0\xi, \ B_2^{n-i-1} + t_1\xi \right) ,
\end{align}
where $B_2^i$ is still the unit Euclidean ball in $F = F_\xi\cap \xi^\perp$ centred at the origin, and $B_2^{n-i-1}$ is the unit Euclidean ball in $F_\xi^\perp$ centred at the origin. Note that $L_0$ is a convex body in $\mathbb{R}^n$ with
\begin{align*}
( L_0 \cap E_\xi ) | F_\xi = G
\end{align*}
and centroid
\begin{align*}
g(L_0) = -t_0\xi + ( t_1 + t_0 ) \left( \frac{n-i}{n+1} \right) \xi .
\end{align*}
Consider the translate $L := L_0 - g(L_0)$ with centroid at the origin. Assume $\langle g(L_0),\xi\rangle < 0$. In this case,
\begin{align*}
&\mbox{vol}_{i+1} \Big( \big( ( L \cap E_\xi ) | F_\xi \big) \cap\xi^- \Big) < \mbox{vol}_{i+1} \big( G\cap\xi^- \big) \\
\mbox{and} \quad &\mbox{vol}_{i+1} \big( G\cap\xi^+ \big) < \mbox{vol}_{i+1} \Big( \big( ( L \cap E_\xi ) | F_\xi \big) \cap\xi^+ \Big) ,
\end{align*}
implying
\begin{align}\label{contra2}
\frac{ \mbox{vol}_{i+1} \big( G\cap\xi^+ \big) }{ \mbox{vol}_{i+1} \big( G\cap\xi^- \big) }
	< \frac{ \mbox{vol}_{i+1} \Big( \big( ( L \cap E_\xi ) | F_\xi \big) \cap\xi^+ \Big) }
	{ \mbox{vol}_{i+1} \Big( \big( ( L \cap E_\xi ) | F_\xi \big) \cap\xi^- \Big)} .
\end{align}
Observe that $L$ gives equality in Proposition \ref{Gruenbaum}. Considering Remark \ref{Grem}, and combining inequalities (\ref{contra1}) and (\ref{contra2}), we get
\begin{align*}
\frac{ \mbox{vol}_{i+1} \Big( \big( (K \cap E_\xi)| F_\xi \big) \cap\xi^+ \Big) }
	{ \mbox{vol}_{i+1} \Big( \big( (K \cap E_\xi)| F_\xi \big) \cap\xi^- \Big) }
	 < \frac{ \mbox{vol}_{i+1} \Big( \big( ( L \cap E_\xi ) | F_\xi \big) \cap\xi^+ \Big) }
	{ \mbox{vol}_{i+1} \Big( \big( ( L \cap E_\xi ) | F_\xi \big) \cap\xi^- \Big)}	
	 = \frac{ i^i }{ (n+1)^i - i^i },
\end{align*}
which contradicts inequality (\ref{inequality1}) in Proposition \ref{Gruenbaum}. Therefore, $\langle g(L_0),\xi\rangle \geq 0$, and consequently
\begin{align}\label{average3}
\mbox{vol}_i \Big( \big( ( L \cap E_\xi ) | F_\xi \big) \cap F \Big)
	\leq \mbox{vol}_i \Big( \big( (K \cap E_\xi)| F_\xi \big) \cap F \Big) .
\end{align}
An explicit calculation gives
\begin{align}\label{average4}
\mbox{vol}_i \Big( \big( ( L \cap E_\xi ) | F_\xi \big) \cap F \Big)
	&= \left( \frac{ i+1 }{ n+1 } \right)^i
	\max_{t\in\mathbb{R}} \mbox{vol}_i \Big( \big( ( L \cap E_\xi ) | F_\xi \big) \cap \lbrace t\xi+ F \rbrace \Big) \nonumber \\
&= \left( \frac{ i+1 }{ n+1 } \right)^i \mbox{vol}_i \Big( \big( (K \cap E_\xi)| F_\xi \big) \cap \lbrace -t_0\xi+ F \rbrace \Big) .
\end{align}

Inequality (\ref{main ineq}) in our theorem statement follows from inequalities (\ref{average3}) and (\ref{average4}), together with the integral expressions (\ref{average1}) and (\ref{average2}). We still need to show that the constant in inequality (\ref{main ineq}) is the best possible. To this end, assume there is a constant
\begin{align*}
C > \left( \frac{i+1}{n+1} \right)^i
\end{align*}
such that
\begin{align*}
V_i\big( K \cap E \big) \geq C\cdot \max_{x\in E^\perp} V_i\big( (K-x) \cap E \big)
\end{align*}
for every convex body $K$ in $\mathbb{R}^n$ with centroid at the origin. Define the convex bodies
\begin{align*}
K_\epsilon = \mbox{conv} \left( B_2^i - \left( \frac{n-i}{n+1}\right) \xi,
	\epsilon B_2^{n-i-1} + \left( \frac{i+1}{n+1} \right) \xi \right)
\end{align*}
for $\epsilon >0$ and fixed $F\in G(E,i)$, with $B_2^i\subset F$ and $B_2^{n-i-1}\subset F_\xi^\perp$ as before. Note that $g(K_\epsilon) = o$ for all $\epsilon>0$. We have
\begin{align*}
\lim_{\epsilon\rightarrow 0^+} K_\epsilon \cap E = \left( \frac{i+1}{n+1} \right) B_2^i
	\quad \mbox{and} \quad &\lim_{\epsilon\rightarrow 0^+}
	\left( K_\epsilon + \left( \frac{n-i}{n+1}\right) \xi \right) \cap E = B_2^i ,
\end{align*}
where the convergence is with respect to the Hausdorff metric. Therefore, it follows from our assumption and the continuity of intrinsic volumes that
\begin{align*}
V_i \left(  \left( \frac{i+1}{n+1} \right) B_2^i \right)
	&= \lim_{\epsilon\rightarrow 0^+} V_i\big( K_\epsilon \cap E \big) \\
&\geq C \lim_{\epsilon\rightarrow 0^+} \max_{x\in E^\perp} V_i\big( (K_\epsilon-x) \cap E \big) \\
&\geq C \lim_{\epsilon\rightarrow 0^+} V_i\left( \left( K_\epsilon + \left( \frac{n-i}{n+1}\right) \xi \right) \cap E \right)
	= C \cdot V_i(B_2^i) .
\end{align*}
This implies
\begin{align*}
\left( \frac{i+1}{n+1} \right)^i \geq C ,
\end{align*}
a contradiction.
\end{proof}

\begin{rem*}
It is often convenient to state inequalities with dimension-independent constants. For example, the classic Gr\"unbaum inequality (\ref{Gruenbaum_classic}) yields
\begin{align*}
\mbox{vol}_n (K \cap \xi^+ ) \geq e^{-1} \mbox{vol}_n (K)
	\quad \mbox{for all} \quad \xi\in S^{n-1},
\end{align*}
for any convex body $K\subset \mathbb R^n$ with centroid at the origin. The constant $e^{-1}$ is asymptotically sharp. Similarly, inequality (\ref{MM Conjecture}) for $k=n-1$ gives
\begin{align*}
\mbox{vol}_{n-1} (K\cap E ) \geq e^{-1} \max_{x\in\mathbb{R}^n}
	\mbox{vol}_{n-1} \big( (K-x)\cap E \big)
\end{align*}
for every $(n-1)$-dimensional subspace $E$.

Our Theorem \ref{MainThm1} yields an analogue of the previous inequality for surface areas of sections with $(n-1)$-dimensional subspaces $E$:
\begin{align*}
V_{n-2} (K\cap E ) \geq e^{-2}
	\max_{x\in K} V_{n-2} \big( (K-x) \cap E \big) .
\end{align*}
The constant $e^{-2}$ is asymptotically sharp.
\end{rem*}

\section{Dual Volumes of Sections}

\begin{proof}[Proof of Theorem \ref{MainThm2}]

The proof proceeds similarly to that of Theorem \ref{MainThm1}. By continuity, there is an $x_0\in K$ such that
\begin{align*}
\widetilde{V}_i \big( ( K - x_0 ) \cap E \big) = \max_{x\in K} \widetilde{V}_i \big( ( K - x ) \cap E \big) .
\end{align*}
Without loss of generality, $x_0$ is not the origin. Note that
\begin{align*}
\Big( (K-x_0) \cap E \Big) \cap F = \Big( \big( K\cap F_{x_0} \big) \cap \lbrace x_0 + F\rbrace \Big) - x_0
	\qquad \forall \ F\in G(E,i).
\end{align*}
By the dual Kubota integral formula,
\begin{align}
\widetilde{V}_i \big( ( K - x_0 ) \cap E \big)
	&= c_{k,i} \int_{G(E,i)} \mbox{vol}_i \Big( \big( K \cap F_{x_0} \big) \cap \lbrace x_0 + F \rbrace \Big) \, dF \label{average5} \\
\mbox{and} \qquad \widetilde{V}_i(K\cap E)
	&= c_{k,i} \int_{G(E,i)} \mbox{vol}_i \Big( \big( K \cap F_{x_0} \big) \cap F \Big) \, dF \label{average6} ,
\end{align}
where $c_{k,i}>0$ is a constant depending on $k$ and $i$.

Consider any $F\in G(E,i)$ for which
\begin{align*}
\mbox{vol}_i \Big( \big( K \cap F_{x_0} \big) \cap \lbrace x_0 + F \rbrace \Big)
	> \mbox{vol}_i \Big( \big( K \cap F_{x_0} \big) \cap F \Big) .
\end{align*}
Let $\xi\in S^{n-1}$ be the unique unit vector that is parallel to $x_0$ and such that $-t_0 := \langle \xi,x_0\rangle < 0$. Let $\widetilde{K}$ be the $i$-symmetral of $K\cap F_\xi$ in $F_\xi$ parallel to $F$. Let $G$ be the unique $(i+1)$-dimensional cone in $F_\xi$ satisfying (\ref{cone}); see Figure \ref{RatioCone}. Define $L_0$ as in (\ref{L0}), and $L := L_0 - g(L_0)$ as before. Following the argument in the proof of Theorem \ref{MainThm1}, we find
\begin{align*}
\frac{ \mbox{vol}_{i+1} \Big( \big( K\cap F_\xi \big) \cap \xi^+ \Big) }{ \mbox{vol}_{i+1} \Big( \big( K\cap F_\xi \big) \cap \xi^- \Big) }
	= \frac{ \mbox{vol}_{i+1} \big( \widetilde{K} \cap \xi^+ \big) }{ \mbox{vol}_{i+1} \big( \widetilde{K} \cap \xi^- \big) }
	\leq \frac{ \mbox{vol}_{i+1} \big( G \cap \xi^+ \big) }{ \mbox{vol}_{i+1} \big( G \cap \xi^- \big) } .
\end{align*}
If $\langle g(L_0),\xi\rangle < 0$, we then have
\begin{align*}
\frac{ \mbox{vol}_{i+1} \big( G \cap \xi^+ \big) }{ \mbox{vol}_{i+1} \big( G \cap \xi^- \big) }
	< \frac{ \mbox{vol}_{i+1} \Big( \big( L\cap F_\xi \big) \cap \xi^+ \Big) }
	{ \mbox{vol}_{i+1} \Big( \big( L\cap F_\xi \big) \cap \xi^- \Big) }
	= \frac{ i^i }{ (n+1)^i - i^i } .
\end{align*}
The two previous inequalities together contradict Gr\"unbaum's inequality for sections (\ref{Gruenbaum_sec}) in light of Remark \ref{Grem}, so necessarily $\langle g(L_0),\xi\rangle \geq 0$. Therefore,
\begin{align*}
\mbox{vol}_i \Big( \big( K \cap F_\xi \big) \cap F \Big)
	&\geq \mbox{vol}_i \Big( \big( L \cap F_\xi \big) \cap F \Big) \\
&= \left( \frac{i+1}{n+1} \right)^i \max_{t\in\mathbb{R}} \mbox{vol}_i \Big( \big( L \cap F_\xi \big) \cap \lbrace t\xi + F\rbrace \Big) \\
&= \left( \frac{i+1}{n+1} \right)^i \mbox{vol}_i \Big( \big( K \cap F_\xi \big) \cap \lbrace -t_0\xi + F\rbrace \Big) ;
\end{align*}
this inequality, (\ref{average5}), and (\ref{average6}) imply inequality (\ref{main ineq 2}) in our theorem statement.

We now show that inequality (\ref{main ineq 2}) is sharp. Assume $i<k$; the equality conditions when $i=k$ are described by Fradelizi \cite{F}. Let $E$ be a $k$-dimensional subspace of $\mathbb{R}^n$, let $F$ be an $i$-dimensional subspace of $E$, and let $\xi\in S^{n-1}\cap E^\perp$. For each $\epsilon >0$, define the convex body
\begin{align*}
K_\epsilon := \mbox{conv} \left( \epsilon B_2^i - \left( \frac{n-i}{n+1}\right) \xi, \
	B_2^{n-i-1} + \left( \frac{i+1}{n+1} \right) \xi \right)
\end{align*}
in $\mathbb{R}^n$. Here, $B_2^i$ and $B_2^{n-i-1}$ are the Euclidean balls in $F$ and $F_\xi^\perp$, respectively, with unit radius and center at the origin. The centroid of $K_\epsilon$ is at the origin for all $\epsilon >0$.

Fix
\begin{align*}
-\frac{n-i}{n+1} < t < \frac{i+1}{n+1} .
\end{align*}
We want to calculate the following limit of $i$th order dual volumes taken within $E$:
\begin{align}\label{desired limit}
\lim_{\epsilon\rightarrow 0^+} \frac{ \widetilde{V}_i \big( K_\epsilon \cap E \big) }
{ \widetilde{V}_i \big( ( K_\epsilon - t\xi ) \cap E \big) } .
\end{align}
Notice that $( K_\epsilon - t\xi ) \cap E$ is a Cartesian product of balls. That is,
\begin{align*}
( K_\epsilon - t\xi ) \cap E = \big( (a \epsilon) B_2^i \big) \times \big( b B_2^{k-i} \big)
\end{align*}
where
\begin{align*}
a = a(t) := \frac{i+1}{n+1} - t \qquad \mbox{and}
	\qquad b = b(t) := \frac{n-i}{n+1} + t .
\end{align*}
Using Fubini's theorem and passing to polar coordinates in the balls $ B_2^{i}$ and $B_2^{k-i}$, we have
\begin{align*}
\widetilde{V}_i \big( ( K_\epsilon - t\xi ) \cap E \big)
&= \frac{i}{k} \int_{( K_\epsilon - t\xi ) \cap E} |x|^{-k+i} \, dx \\
&= \frac{i}{k} \int_{\epsilon a B_2^{i}}
	\int_{b B_2^{k-i}} (x_1^2 +\cdots+ x_k^2)^{(-k+i)/2} dx_1 \, \cdots \, dx_k \\
&= \frac{i}{k} \omega_i\, \omega_{k-i} \int_0^{\epsilon a} r_1^{i-1}
	\int_0^{b} r_2^{k-i-1} (r_1^2+r_2^2)^{(-k+i)/2} dr_2 \, dr_1 .
\end{align*}
The notation $\omega_i$ gives the surface area of the $B_2^i$. Denoting
\begin{align*}
a_0 = \frac{i+1}{n+1} \quad \mbox{and} \quad b_0 = \frac{n-i}{n+1},
\end{align*}
we obtain
\begin{align*}
  \frac{ \widetilde{V}_i \big( K_\epsilon \cap E \big) }
{ \widetilde{V}_i \big( ( K_\epsilon - t\xi ) \cap E \big) } =
\frac{ \int_0^{\epsilon a_0} r_1^{i-1}   \int_0^{b_0}  r_2^{k-i-1}  (r_1^2+r_2^2)^{(-k+i)/2} dr_2 \, dr_1}{ \int_0^{\epsilon a} r_1^{i-1}   \int_0^{b}  r_2^{k-i-1}  (r_1^2+r_2^2)^{(-k+i)/2} dr_2 \, dr_1} .
\end{align*}
An application of the Dominated Convergence Theorem verifies that the numerator and denominator above approach zero as $\epsilon$ tends to zero. Thus, with L'H\^opital's rule we obtain
\begin{align*}
\lim_{\epsilon\rightarrow 0^+} \frac{ \widetilde{V}_i \big( K_\epsilon \cap E \big) }
{ \widetilde{V}_i \big( ( K_\epsilon - t\xi ) \cap E \big) } = \lim_{\epsilon\rightarrow 0^+}
\frac{ \epsilon^{i-1} a_0 ^{i }   \int_0^{b_0}  r_2^{k-i-1}  ((\epsilon a_0)^2+r_2^2)^{(-k+i)/2} dr_2  }{\epsilon^{i-1} a ^{i }  \int_0^{b}  r_2^{k-i-1}  ((\epsilon a )^2+r_2^2)^{(-k+i)/2} dr_2 } .
\end{align*}
The integrals in the numerator and denominator both approach infinity as $\epsilon$ tends to zero. We will show that their ratio approaches one. To see this, write the integral in the denominator as the sum of the integrals: from $0$ to $a$ and from $a$ to $b$. As $\epsilon$ approaches zero, the integral from $a$ to $b$ approaches some constant, and so we will disregard it when computing the limit. The same argument applies to the integral in the numerator. Therefore,
\begin{align*}
&\lim_{\epsilon\rightarrow 0^+} \frac{ \widetilde{V}_i \big( K_\epsilon \cap E \big) }
	{ \widetilde{V}_i \big( ( K_\epsilon - t\xi ) \cap E \big) } \\
&= \lim_{\epsilon\rightarrow 0^+}
	\frac{ a_0^i \int_0^{a_0} r_2^{k-i-1} ((\epsilon a_0)^2+r_2^2)^{(-k+i)/2} dr_2 }
	{ a^i \int_0^a r_2^{k-i-1} ((\epsilon a )^2+r_2^2)^{(-k+i)/2} dr_2 }
	= \frac{ a_0^i }{ a^i }
	= \left( \frac{ \frac{i+1}{n+1} }{ \frac{i+1}{n+1} - t } \right)^i ,
\end{align*}
since the integrals in the numerator and denominator are the same, which can be seen using an obvious change of variables.

We finally have
\begin{align*}
&\inf_{ -\frac{n-i}{n+1}<t<\frac{i+1}{n+1} }
	\left( \lim_{\varepsilon\rightarrow 0^+} \frac{ \widetilde{V}_i \big( K_\varepsilon \cap E \big) }
	{ \widetilde{V}_i \big( ( K_\varepsilon - t\xi ) \cap E \big) } \right)
	= \left( \frac{i+1}{n+1} \right)^i ,
\end{align*}
which proves that the constant in (\ref{main ineq 2}) is the best possible.
\end{proof}

\section{Gr\"unbaum's Inequality for Dual Volumes}

\begin{proof}[Proof of Theorem \ref{MainThm3}]

Assume throught the proof that $i<k$, as Theorem \ref{MainThm3} reduces to (\ref{Gruenbaum_sec}) and (\ref{Gruenbaum_pro}) for $i=k$.

For a $k$-dimensional subspace $E\subset\mathbb{R}^n$, $\xi\in S^{n-1}\cap E$, and any $i$-dimensional subspace $F\subset E$, Gr\"unbaum's inequality for sections (\ref{Gruenbaum_sec}) gives
\begin{align*}
\mbox{vol}_i \Big( \big( K\cap E \cap\xi^+ \big)\cap F \Big)
	&= \mbox{vol}_i ( K\cap F \cap \xi^+ ) \\
&\geq \left( \frac{i}{n+1}\right)^i \mbox{vol}_i (K\cap F )
	= \left( \frac{i}{n+1}\right)^i \mbox{vol}_i \big( (K\cap E)\cap F \big) ,
\end{align*}
and inequality (\ref{inequality2}) of Proposition \ref{Gruenbaum} gives
\begin{align*}
\mbox{vol}_i \Big( \big( (K|E)\cap\xi^+ \big) \cap F \Big)
	= \mbox{vol}_i \big( (K|E)\cap F\cap\xi^+ \big)
	\geq \left( \frac{i}{n+1} \right)^i \mbox{vol}_i \big( (K|E)\cap F \big) .
\end{align*}
Integrating these inequalities over $G(E,i)$, and applying the dual Kubota formula, we respectively get
\begin{align*}
\widetilde{V}_i ( K\cap E\cap \xi^+ )
	&\geq \left( \frac{i}{n+1} \right)^i \widetilde{V}_i (K\cap E) \\
\mbox{and} \qquad \widetilde{V}_i \big( (K|E)\cap\xi^+ \big)
	&\geq \left( \frac{i}{n+1} \right)^i \widetilde{V}_i (K|E) .
\end{align*}
The dual volumes are taken within $E$.

We now prove that the constant in the above inequality is the best possible. Fix a $k$-dimensional subspace $E\subset\mathbb{R}^n$, an $i$-dimensional subspace $F\subset E$, and $\xi\in S^{n-1}\cap F$. Consider the family of bodies
\begin{align*}
K_\epsilon  = \mbox{conv} \left( \epsilon B_2^{i-1}
	- \epsilon \left( \frac{n-i+1}{n+1}\right) \xi ,
	B_2^{n-i} + \epsilon \left( \frac{i}{n+1} \right) \xi \right)
\end{align*}
for small $\epsilon>0$. Here, $B_2^{i-1}$ is the unit ball in $F\cap\xi^\perp$, and $B_2^{n-i}$ is the unit ball in $F^\perp$. Interpret $\epsilon B_2^{i-1}$ as the origin when $i=1$. Note that the centroid of $K_\varepsilon$ is at the origin. It is sufficient to show that
\begin{align}\label{reverse}
\lim_{\epsilon\to 0^+}
	\frac{\widetilde{V}_i ( K_\epsilon \cap E \cap \xi^+ )}
	{\widetilde{V}_i  (K_\epsilon \cap E ) } \leq \left( \frac{i}{n+1} \right)^i ,
\end{align}
because $K_\epsilon \cap E = K_\epsilon |E$.

First consider the case $2\leq i\leq k-1$. For each $t\in \Big[ - \epsilon \left( \frac{n-i+1}{n+1}\right), \epsilon \left( \frac{i}{n+1} \right)\Big]$, the section $K_\epsilon \cap E \cap \{ t\xi + \xi^\perp \}$ is the product of two balls
\begin{align*}
\big( a(t) B_2^{i-1} \big) \times \big( b(t) B_2^{k-i} \big)
\end{align*}
where
\begin{align*}
a(t) = \epsilon \frac{i }{n+1} - t \qquad \mbox{and}
	\qquad b(t) = \frac{n-i+1}{n+1} + \frac{t}{\epsilon} .
\end{align*}
Using Fubini's Theorem and passing to polar coordinates in the balls $B_2^{i-1}$ and $B_2^{k-i}$, we have
\begin{align*}
&\widetilde{V}_i ( K_\epsilon\cap E \cap \xi^+ )
	=\frac{i}{k} \int_{  K_\epsilon\cap E \cap \xi^+} |x|^{-k+i} dx \\
&= \frac{i}{k}\int_{0}^{\epsilon \left( \frac{i}{n+1} \right)}
	\int_{a(x_1)  B_2^{i-1}} \int_{b(x_1) B_2^{k-i}}
	(x_1^2 +\cdots+ x_k^2)^{(-k+i)/2} dx_k \, \cdots \, dx_1 \\
&= \frac{i}{k} \omega_{i-1} \omega_{k-i} \\
&\times	\int_{0}^{\epsilon \left( \frac{i}{n+1} \right)} \int_0^{a(x_1)} r_1^{i-2}
	\int_0^{b(x_1)} r_2^{k-i-1}  (x_1^2 + r_1^2+r_2^2)^{(-k+i)/2} dr_2 \, dr_1\, dx_1
\end{align*}
Making the change of variables $x_1 = \epsilon u$, $r_1 = \epsilon v$, $r_2 = w$, we get
\begin{align*}
&\widetilde{V}_i ( K_\epsilon\cap e_1^+ ) = \frac{i}{k} \omega_{i-1} \omega_{k-i}
	\epsilon^i \\
&\qquad \times \int_{0}^{ \frac{i}{n+1} } \int_0^{\frac{i}{n+1} - u } v^{i-2}   		
	\int_0^{ \frac{n-i+1}{n+1} + u}  w^{k-i-1}
	(\epsilon^2 u^2 + \epsilon^2 v^2+w^2)^{(-k+i)/2}\, dw \, dv\, du.
\end{align*}
Denoting the latter triple integral by $I$, and denoting by $II$ the triple integral below
\begin{align*}
II = \int_{-\frac{n-i+1}{n+1}}^{ \frac{i}{n+1} } \int_0^{\frac{i }{n+1} - u } v^{i-2}
	\int_0^{ \frac{n-i+1}{n+1} + u}  w^{k-i-1}
	(\epsilon^2 u^2 + \epsilon^2 v^2+w^2)^{(-k+i)/2}\, dw \, dv\, du,
\end{align*}
we see that
\begin{align*}
\frac{\widetilde{V}_i ( K_\epsilon \cap  E \cap \xi^+ )}
	{\widetilde{V}_i (K_\epsilon \cap E ) } = \frac{I}{II}.
\end{align*}

We estimate $I$ from above by
\begin{align*}
I &\leq \int_{0}^{ \frac{i}{n+1} } \int_0^{\frac{i }{n+1} - u } v^{i-2}
	\int_0^{ 1} w^{k-i-1} (\epsilon^2 u^2 +w^2)^{(-k+i)/2}\, dw \, dv\, du .
\end{align*}
If $i=k-1$, then the integral with respect to $w$ can be computed directly and it equals $\ln(1+\sqrt{1+\epsilon^2 u^2}) - \ln(\epsilon u) $. If $i\leq k-2$, then
\begin{align*}
\frac{ w^{k-i-1} }{ (\epsilon^2 u^2 + w^2)^\frac{k-i}{2} }
	\leq \frac{w}{ \epsilon^2 u^2 + w^2 } ,
\end{align*}
and so
\begin{align*}
I &\leq \int_{0}^{ \frac{i}{n+1} } \int_0^{\frac{i }{n+1} - u } v^{i-2} \, dv
	\int_0^{ 1} w (\epsilon^2 u^2 +w^2)^{-1}\, dw \, du\\
&= \frac{ 1 }{ 2(i-1) } \int_{0}^{ \frac{i}{n+1} }
	\left(\frac{i }{n+1} - u \right)^{i-1}
	\left( \ln(\epsilon^2u^2+1) - 2\ln(\epsilon) - 2\ln(u) \right) \, du \\
&= o(1) -\frac{ \ln\epsilon }{ i-1 } \int_{0}^{ \frac{i}{n+1} }
	\left(\frac{i }{n+1} - u \right)^{i-1} \, du
	-\int_{0}^{ \frac{i}{n+1} }
	\left(\frac{i }{n+1} - u \right)^{i-1} \frac{\ln(u)}{i-1} \, du \\
&= -\frac{ \ln\epsilon }{i(i-1)} \left(\frac{i}{n+1} \right)^{i} (1 +o(1)).
\end{align*}
Note that in the case $i=k-1$, we get the same bound.

Now we estimate $II$ from below. Since $|u|\leq 1$, $|v|\leq 1$, we have
\begin{align*}
II &\geq \int_{-\frac{n-i+1}{n+1}}^{ \frac{i}{n+1} } \int_0^{\frac{i }{n+1} - u }
	v^{i-2} \int_0^{ \frac{n-i+1}{n+1} + u} w^{k-i-1}
	(2\epsilon^2 +w^2)^{(-k+i)/2} \, dw \, dv\, du \\
&= \frac{1}{i-1} \int_{-\frac{n-i+1}{n+1}}^{ \frac{i}{n+1} }
	\left( \frac{i }{n+1} - u \right)^{i-1} \int_0^{ \frac{n-i+1}{n+1} + u}  w^{k-i-1}
	(2\epsilon^2 +w^2)^{(-k+i)/2} \, dw\, du .
\end{align*}
Using the change of variable $z = \frac{n-i+1 }{n+1} + u $ and then integrating by parts, we get
\begin{align*}
&= \frac{1}{i-1} \int_0^1 (1-z)^{i-1}
	\int_0^z w^{k-i-1} (2\epsilon^2 +w^2)^{(-k+i)/2} \, dw\, dz \\
&= \frac{1}{i(i-1)} \int_0^1 (1-z)^i z^{k-i-1}
	\left( 2\epsilon^2 +z^2\right)^{(-k+i)/2} \, dz \\
&\geq \frac{1}{i(i-1)} \int_\epsilon^1 (1-z)^i z^{k-i-1}
	\left( 2\epsilon^2 +z^2\right)^{(-k+i)/2} \, dz \\
&= \frac{1}{i(i-1)} \int_\epsilon^1 (1-z)^i z^{-1}
	\left( 2\epsilon^2 z^{-2} +1 \right)^{(-k+i)/2} \, dz \\
&\geq \frac{1}{i(i-1)} \int_\epsilon^1 (1-z)^i z^{-1}
	\left( 1+ (-k+i) \epsilon^2 z^{-2} \right) \, dz,
\end{align*}
where we  used the inequality $(1+x)^p \ge 1+ px$ with $p<0$ and $x\ge 0$. Note that
\begin{align*}
(k-i)\, \epsilon^2 \int_\epsilon^1 (1-z)^i z^{-3} \, dz
\end{align*}
is positive, and bounded above by a constant $C>0$ for small enough $\epsilon>0$. Thus,
\begin{align*}
II &\geq \frac{1}{i(i-1)} \int_\epsilon^1 (1-z)^i z^{-1} \, dz - C \\
&= \frac{1}{i(i-1)} \int_\epsilon^1
	\left( z^{-1} + \sum_{j=1}^i \genfrac(){0pt}{0}{i}{j} (-1)^j z^{j-1} \right)
	\, dz - C \\
&= - \frac{ \ln \epsilon}{i(i-1)} (1+o(1)).
\end{align*}
Comparing the bounds for $I$ and $II$, we get (\ref{reverse}).

We now consider the case $i=1$, in which $K_\epsilon$ is a cone. The section
\begin{align*}
K_\epsilon \cap E \cap \{ t\xi+\xi^\perp \}
\end{align*}
is a ball $b(t) B^{n-1}$ for each $t\in \Big[ - \epsilon \left( \frac{n}{n+1}\right), \epsilon \left( \frac{1}{n+1} \right)\Big]$, with
\begin{align*}
b(t) = \frac{n}{n+1} + \frac{t}{\epsilon} .
\end{align*}
Using Fubini's Theorem, polar coordinates, and a change of variables, we find
\begin{align*}
\frac{ \widetilde{V}_1 ( K_\epsilon\cap E\cap \xi^+ ) }
	{ \widetilde{V}_1 ( K_\epsilon\cap E ) }
	= \frac{ \int_0^{ \frac{1}{n+1} } \int_0^{ \frac{n}{n+1} + u} w^{k-2}
	\left( \epsilon^2 u^2 + w^2 \right)^{(-k+1)/2} \, dw \, du }
	{ \int_{- \frac{n}{n+1} }^{ \frac{1}{n+1} } \int_0^{ \frac{n}{n+1} + u} w^{k-2}
	\left( \epsilon^2 u^2 + w^2 \right)^{(-k+1)/2} \, dw \, du } .
\end{align*}
The numerator and denominator can again be bounded using the previous methods, so that we obtain (\ref{reverse}).
\end{proof}

\end{document}